\documentclass[11pt]{amsart}
\usepackage[T1]{fontenc}
\usepackage[utf8]{inputenc}
\usepackage[UKenglish]{babel}
\usepackage[margin=1in]{geometry}
\usepackage{lmodern}
\usepackage{microtype}
\usepackage{amssymb,mathtools,amsthm}
\usepackage{tikz-cd}
\usepackage{enumitem}
\usepackage{hyperref}
\hypersetup{hidelinks,pdftitle={Étale motives with and without transfers},pdfauthor={Bo Zhang}}
\numberwithin{equation}{section}
\newtheorem{theorem}{Theorem}[section]
\newtheorem{proposition}[theorem]{Proposition}
\newtheorem{lemma}[theorem]{Lemma}
\newtheorem{corollary}[theorem]{Corollary}
\theoremstyle{definition}

\theoremstyle{remark}
\newtheorem{remark}[theorem]{Remark}
\title[\'Etale motives with and without transfers]{\'Etale motives with and without transfers}
\author{Bo Zhang}
\address{School of Mathematical Sciences, Capital Normal University, 105 Xisanhuan North Road, Haidian District, 100048 Beijing, China}
\email{bozhang9511@gmail.com}
\date{}
\begin{document}
\begin{abstract}
We compare stable \'etale motives defined with and without finite transfers. For every noetherian scheme $S$ of finite Krull dimension and every commutative ring $R$, the canonical functor
$\mathrm{DA}_{\mathrm{\acute{e}t}}(S,R)\to\mathrm{DM}_{\mathrm{\acute{e}t}}(S,R)$
is an equivalence of symmetric monoidal stable $\infty$-categories. No excellence, geometric unibranchness, or restriction on the residue characteristics is imposed on $S$, and no flatness or finiteness condition is imposed on $R$. The rational comparison is proved by relating divided effective-cycle monoids on split tori to geometric symmetric powers. For wild permutation actions, coordinate power maps and invariant symmetric polynomials give a vanishing theorem after passage to arithmetic neighbourhoods and sequential telescopes. A homotopy of roots-of-unity cycles removes these telescopes on the divided cycle spaces. Seminormal descent and weak continuity extend the comparison to general bases. Bachmann's rigidity theorem and the rigidity theorem with transfers then give the comparison with finite-field coefficients. Compatibility of the forgetful right adjoint with rationalization allows these comparisons to be assembled integrally, after which extension of coefficients gives the result for $R$. As a consequence, rational motivic cohomology defined by finite correspondences is the rational plus sphere and commutes with arbitrary base change in this class of schemes.
\end{abstract}
\maketitle
\tableofcontents
\enlargethispage{4pt}

\section{Introduction}\label{sec:introduction}
Finite correspondences provide the transfer maps in Voevodsky's construction of motives. A second construction starts with sheaves on smooth schemes and imposes homotopy invariance and Tate stability without adjoining correspondences. With the \'etale topology, these constructions are expected to encode the same theory. Their comparison involves two distinct ingredients: rigidity controls torsion coefficients, whereas rational coefficients require a geometric comparison between points and finite relative cycles. This paper establishes that geometric comparison over arbitrary noetherian schemes of finite dimension and uses it to obtain the \'etale comparison for all commutative coefficient rings.

Let $S$ be a noetherian scheme of finite Krull dimension and let $R$ be a commutative ring with unit. We write $\mathrm{DA}_{\mathrm{\acute{e}t}}(S,R)$ for the stable \'etale $\mathbb A^1$-derived category without transfers, and $\mathrm{DM}_{\mathrm{\acute{e}t}}(S,R)$ for its counterpart formed from \'etale sheaves with finite transfers. We use stable $\infty$-categorical enhancements and \'etale hyperdescent throughout. Sending a morphism of smooth $S$-schemes to its graph defines the symmetric monoidal comparison functor
\[
 F_{S,R}:\mathrm{DA}_{\mathrm{\acute{e}t}}(S,R)
 \longrightarrow\mathrm{DM}_{\mathrm{\acute{e}t}}(S,R).
\]

\begin{theorem}[\'Etale comparison]\label{thm:main}
For every noetherian scheme $S$ of finite Krull dimension and every commutative ring $R$, the functor $F_{S,R}$ is an equivalence of symmetric monoidal stable $\infty$-categories. These equivalences are compatible with pullback in $S$ and with extension of coefficients in $R$.
\end{theorem}

The theorem applies in equal and mixed characteristics, including singular and non-geometrically-unibranch bases. It requires neither finite \'etale cohomological dimension of the residue fields nor flatness of $R$ over $\mathbb Z$. The notation $\mathrm{DM}_{\mathrm{\acute{e}t}}$ always refers here to the category constructed using finite correspondences on the relative smooth site. In particular, the target is specified independently of the comparison theorem.

\subsection*{Earlier comparisons and rigidity}
Cisinski and D\'eglise identified rational motives with transfers with Beilinson motives over excellent geometrically unibranch bases \cite[Theorem~16.1.4]{CD}, and identified rational \'etale motives without transfers with Beilinson motives over general finite-dimensional noetherian bases \cite[Theorem~16.2.18]{CD}. Their theory of \'etale motives compares the finite-correspondence construction with $h$-motives over quasi-excellent geometrically unibranch bases for arbitrary coefficients \cite[Corollary~5.5.5]{CDEt}. The comparison between motives without transfers and $h$-motives is treated in \cite[Corollary~5.5.7 and Remarks~5.5.8--5.5.9]{CDEt}. The rational finite-correspondence comparison over bases without a unibranchness hypothesis is the geometric issue addressed below.

For torsion coefficients, rigidity gives a different description. If $p$ is invertible on $S$, Bachmann's rigidity theorem identifies \'etale motives without transfers with $\mathbb F_p$-complexes on the small \'etale site \cite[Corollary~3.2]{BachmannEtale}. The corresponding statement with transfers is the rigidity theorem of Cisinski and D\'eglise \cite[Theorem~4.5.2]{CDEt}. For an arbitrary base, both $\mathbb F_p$-linear motivic categories depend only on $S[1/p]$; this follows from the Artin--Schreier sequence and the argument of \cite[Proposition~A.3.4]{CDEt}. These results identify the canonical comparison functor with an equivalence for every prime $p$, including $p=2$.

To combine the rational and torsion comparisons, we use an arithmetic detection lemma for linear stable categories. Its application requires compatibility of the comparison adjunction with the relevant coefficient extensions. Exactness handles reduction modulo $p$. For rationalization, we prove the needed compatibility directly using \'etale hypercohomology and Tate $\Omega$-spectra. This avoids imposing compact-generation hypotheses on the integral \'etale categories. Once the integral comparison is established, passage to $R$-module objects gives Theorem~\ref{thm:main} for an arbitrary commutative ring $R$.

\subsection*{The rational comparison}
We retain the notation $\mathrm{DM}(S,\mathbb Q)$ for stable Nisnevich motives with rational finite transfers, and $\mathrm{DA}(S,\mathbb Q)$ for the rational $\mathbb A^1$-derived category. Rational Nisnevich and \'etale motives with transfers agree \cite[Proposition~2.2.10]{CDEt}. The rational theorem needed for Theorem~\ref{thm:main} therefore takes the following form.

\begin{theorem}[Rational comparison]\label{thm:rational-main}
For every noetherian scheme $S$ of finite dimension, the canonical functor
\[
 \mathrm{DA}_{\mathrm{\acute{e}t}}(S,\mathbb Q)
 \longrightarrow\mathrm{DM}(S,\mathbb Q)
\]
is an equivalence. There are consequently canonical symmetric monoidal equivalences
\[
 \mathrm{DM}(S,\mathbb Q)
 \simeq\mathrm{DM}_{\mathrm{\acute{e}t}}(S,\mathbb Q)
 \simeq\mathrm{DM}_{\mathrm B}(S)
 \simeq\mathrm{SH}(S)_{\mathbb Q,+}.
\]
\end{theorem}

Here $\mathrm{DM}_{\mathrm B}(S)$ denotes Beilinson motives, and $\mathrm{SH}(S)_{\mathbb Q,+}$ is the plus part of the rational motivic stable homotopy category. The rational $\mathbb A^1$-derived category satisfies $\mathrm{DA}(S,\mathbb Q)\simeq\mathrm{SH}(S)_{\mathbb Q}$ and also contains the minus part. Keeping the two topologies in the notation is therefore essential.

Let $\varphi_S^*:\mathrm{SH}(S)\to\mathrm{DM}(S,\mathbb Q)$ add rational transfers and let $\varphi_{S*}$ be its right adjoint. Define
\[
 H\mathbb Q_S=\varphi_{S*}(\mathbf1_S),
 \qquad
 \Phi_S:\mathbb S_{S,\mathbb Q,+}\longrightarrow H\mathbb Q_S,
\]
where the second map is the canonical comparison. This is the motivic cohomology spectrum defined by finite correspondences \cite[Definition~11.2.17]{CD}.

\begin{corollary}[The rational motivic cohomology spectrum]\label{thm:ring-main}
For every noetherian scheme $S$ of finite dimension, $\Phi_S$ is an equivalence of commutative ring spectra, and
\[
 \operatorname{Mod}_{H\mathbb Q_S}(\mathrm{SH}(S))
 \simeq\mathrm{DM}(S,\mathbb Q).
\]
For every morphism $f:T\to S$ between such schemes, the canonical exchange map
\[
 f^*H\mathbb Q_S\longrightarrow H\mathbb Q_T
\]
is an equivalence of commutative ring spectra.
\end{corollary}

The base-change statement is the rational form, in this class of bases, of the conjecture in \cite[Conjecture~17]{VoevOpen}; see also \cite[Conjecture~11.2.22]{CD}. The definition of $H\mathbb Q_S$ is intrinsic to the finite-correspondence construction, so its compatibility with base change is a consequence of the comparison.

\subsection*{The geometric argument}
The rational spectrum comparison is reduced to the graph maps for split tori
\begin{equation}\label{eq:intro-graph}
 M_S(\mathbb Q[X])\longrightarrow M_S(L_S(X)_{\mathbb Q}),
 \qquad X=\mathbb G_{m,S}^{r}.
\end{equation}
The functor $M_S$ takes the motive of an underlying space, while $L_S(X)_{\mathbb Q}$ is the sheaf of rational finite relative cycles. Over a seminormal base, division of multiplicities identifies effective relative cycles with the divisible localization of geometric symmetric powers. After group completion, the graph comparison follows from a comparison between free commutative monoids on $X$ and these geometric symmetric powers.

Wild permutation actions are treated by studying the cofiber of the quotient comparison. For each prime $p$, coordinate power maps act after restriction to the characteristic-$p$ fibre and localization at open neighbourhoods of that fibre. Elementary symmetric polynomials define an equivariant interpolation from a suitable power map to a map through a maximal fixed locus. Induction on stabilizers proves that the resulting sequential telescope vanishes. A separate homotopy of roots-of-unity cycles makes the power map invertible on divided effective-cycle spaces. It is this invertibility, after division of multiplicities, that removes the telescope.

These constructions take place on the relative smooth site. In particular, for a possibly singular quotient $W$, the motive used is that of the restricted representable $U\mapsto\operatorname{Hom}_S(U,W)$. Restriction to a closed characteristic fibre is applied to that motive; it is not replaced by the intrinsic motive of the corresponding quotient over the fibre. This keeps the argument independent of the base-change statement that it eventually proves.

Finite seminormalization gives the rational comparison over finite-type $\mathbb Z$-schemes. Weak continuity along affine dominant transition morphisms then extends the categorical comparison to all noetherian bases of finite dimension. The spectrum comparison and its base-change compatibility follow by evaluating the right adjoint on the unit. Finally, the coefficient argument combines this rational theorem with rigidity and proves Theorem~\ref{thm:main}.

\subsection*{Organization and conventions}
Section~\ref{sec:coefficients} develops the coefficient formalism, verifies rationalization compatibility, proves the torsion comparison, and reduces Theorem~\ref{thm:main} to Theorem~\ref{thm:rational-main}. Section~\ref{sec:formal} gives the split-torus criterion for the rational comparison. Section~\ref{sec:cycles} studies effective cycles, group completion, and power maps. Sections~\ref{sec:quotients} and~\ref{sec:arithmetic} establish the tame and arithmetic permutation comparisons. Section~\ref{sec:graph} proves the graph comparison over finite-type arithmetic bases. Section~\ref{sec:conclusion} completes the rational argument by continuity and deduces the general \'etale theorem.

Unless a limit presentation is explicitly under discussion, schemes are noetherian and of finite Krull dimension. All tensor products of complexes are derived, and all categories, adjunctions, and colimits are understood in their stable $\infty$-categorical sense. The coefficient rings are ordinary commutative rings, but relative tensor products of categories are taken over the derived category $D(\mathbb Z)$.

\section{Coefficients and \'etale rigidity}\label{sec:coefficients}
This section isolates the passage from the rational comparison to the comparison with arbitrary coefficients. We first fix the coefficient conventions, then verify the compatibility required by arithmetic detection. The rational comparison itself is proved in Sections~\ref{sec:formal}--\ref{sec:conclusion}.

\subsection{Models and extension of coefficients}
Let $\mathrm{Sm}_S$ be the category of separated smooth $S$-schemes of finite type. The category $\mathrm{DA}_{\mathrm{\acute{e}t}}(S,R)$ is obtained from derived $R$-linear presheaves on $\mathrm{Sm}_S$ by imposing \'etale hyperdescent and $\mathbb A^1$-invariance, and then inverting the Tate suspension object
\[
 T_R=R(\mathbb A^1_S)/R(\mathbb A^1_S\setminus0)
 \simeq\widetilde R[(\mathbb G_{m,S},1)][1].
\]
The quotient and the reduced linearization in this formula are derived cofibers; the displayed equivalence holds in the effective $\mathbb A^1$-localized category. Equivalently, one can use complexes of \'etale sheaves followed by $\mathbb A^1$-localization and Tate stabilization.

For the transferred construction, replace $\mathrm{Sm}_S$ by its additive category of finite correspondences, formed using finite relative cycles in the sense of Suslin--Voevodsky. An \'etale sheaf with transfers is an additive presheaf whose restriction along the graph functor is an \'etale sheaf. Localizing complexes of such sheaves and stabilizing with the image of $T_R$ gives $\mathrm{DM}_{\mathrm{\acute{e}t}}(S,R)$. The graph functor induces an adjunction
\begin{equation}\label{eq:etale-adjunction}
 F_{S,R}:\mathrm{DA}_{\mathrm{\acute{e}t}}(S,R)
 \rightleftarrows\mathrm{DM}_{\mathrm{\acute{e}t}}(S,R):G_{S,R}.
\end{equation}
Its left adjoint is symmetric monoidal and preserves colimits. Both categories are presentable and stable. Pullback, smooth pushforward, and extension of coefficients are compatible with the graph functor. We use the hypercomplete \'etale models in the rigidity statements below; no finite \'etale cohomological-dimension assumption is implicit in the notation.

\begin{lemma}[Extension of coefficients]\label{lem:coefficient-extension}
Let $\mathcal M$ denote either $\mathrm{DA}_{\mathrm{\acute{e}t}}$ or $\mathrm{DM}_{\mathrm{\acute{e}t}}$. For every commutative ring $R$, there are natural equivalences
\begin{equation}\label{eq:coefficient-extension}
 \mathcal M(S,R)
 \simeq\operatorname{Mod}_{R}\bigl(\mathcal M(S,\mathbb Z)\bigr)
 \simeq\mathcal M(S,\mathbb Z)
       \otimes_{D(\mathbb Z)}D(R).
\end{equation}
Here $R$ denotes the image of the commutative algebra $R\in D(\mathbb Z)$ under the coefficient action. Under these equivalences, $F_{S,R}$ is the scalar extension of $F_{S,\mathbb Z}$.
\end{lemma}
\begin{proof}
For derived presheaves, with or without transfers, extension of coefficients is the ordinary derived extension of scalars. The same description holds for module objects: an $R$-module structure is precisely an $R$-linear structure on the presheaf of complexes. The relations imposing \'etale hyperdescent are the maps associated with \'etale hypercovers. The relations imposing homotopy invariance are the maps associated with $X\times\mathbb A^1\to X$. In each case the $R$-linear relation is obtained from the integral one by extension of scalars.

The universal property of a presentable localization therefore identifies the scalar extension of the integral localized category with the corresponding $R$-linear localized category. This can also be seen on module objects: the localization is $D(\mathbb Z)$-linear, and the local $R$-modules are those whose underlying integral objects satisfy the same descent and homotopy conditions. Finally, Tate stabilization commutes with extension of scalars, since the suspension object $T_R$ is the image of $T_{\mathbb Z}$ and inversion has the same universal property after scalar extension. This proves \eqref{eq:coefficient-extension}. The graph functor commutes with each of these constructions, proving the last assertion.
\end{proof}

In particular, reduction modulo $p$ always means derived tensoring with $\mathbb F_p$. For coefficients $R$ that are not flat over $\mathbb Z$, the derived algebra $R\otimes_{\mathbb Z}^{\mathbf L}\mathbb F_p$ need not be the discrete ring $R/pR$. We will prove the comparison over $\mathbb Z$ first and only then apply \eqref{eq:coefficient-extension}.

\subsection{Arithmetic detection of an equivalence}
For a presentable $D(\mathbb Z)$-linear stable category $\mathcal C$ and a commutative ring $A$, put
\[
 \mathcal C_A=\mathcal C\otimes_{D(\mathbb Z)}D(A).
\]
The notation $x\otimes A$ denotes extension of coefficients; after forgetting its $A$-module structure, it is the action of the complex $A$ on $x$.

\begin{lemma}[Arithmetic detection]\label{lem:arithmetic-detection}
Let $F:\mathcal C\to\mathcal D$ be a $D(\mathbb Z)$-linear colimit-preserving functor between presentable stable categories, with right adjoint $G$. Suppose that the natural map
\begin{equation}\label{eq:abstract-rational-projection}
 G(y)\otimes\mathbb Q\longrightarrow G(y\otimes\mathbb Q)
\end{equation}
is an equivalence for every $y\in\mathcal D$. If $F_{\mathbb Q}$ and $F_{\mathbb F_p}$ are equivalences for every prime $p$, then $F$ is an equivalence.
\end{lemma}
\begin{proof}
A right adjoint between stable categories is exact. For every prime $p$, exactness identifies
\begin{align*}
 G(y)\otimes\mathbb F_p
 &\simeq\operatorname{cofib}\bigl(p:G(y)\to G(y)\bigr)\\
 &\simeq G\bigl(\operatorname{cofib}(p:y\to y)\bigr)
 \simeq G(y\otimes\mathbb F_p).
\end{align*}
Together with \eqref{eq:abstract-rational-projection}, this gives the projection formula for the coefficients $A=\mathbb Q$ and $A=\mathbb F_p$.

The adjunction $F\dashv G$ induces an adjunction on $A$-module objects. Its right adjoint has underlying functor $G$, with the $A$-action induced by the linear adjunction. The projection formulas just proved identify the scalar extensions of the original unit and counit with the unit and counit of this adjunction, evaluated on free $A$-module objects. Since $F_A$ is an equivalence, these scalar-extended unit and counit maps are equivalences.

Let $K$ be the cofiber of any component of the original unit or counit. We have
\[
 K\otimes\mathbb Q=0,
 \qquad K\otimes\mathbb F_p=0\quad\text{for every prime }p.
\]
The second assertion says that multiplication by every prime is an equivalence on $K$. Hence every positive integer acts invertibly. Rationalization is the sequential colimit along multiplication by $2,3,4,\ldots$, so the canonical map $K\to K\otimes\mathbb Q$ is an equivalence. The first assertion now gives $K=0$. Both the unit and the counit are equivalences, as required.
\end{proof}

\begin{remark}\label{rem:filtered-colimit-criterion}
The hypothesis \eqref{eq:abstract-rational-projection} holds whenever $G$ preserves filtered colimits. Indeed, exactness gives preservation of finite coproducts, and arbitrary coproducts are filtered colimits of finite ones. Thus $G$ preserves all colimits, in particular the colimit defining rationalization. For our application, only the displayed rational projection formula is needed; we verify it directly rather than appealing to compact generation of integral \'etale motives.
\end{remark}

\subsection{Rationalization and forgetting transfers}
We now check the hypothesis of Lemma~\ref{lem:arithmetic-detection} for \eqref{eq:etale-adjunction}. The finite-dimensional noetherian hypothesis enters through the following compatibility of unbounded \'etale hypercohomology with rationalization:
\begin{equation}\label{eq:etale-rational-sections}
 R\Gamma_{\mathrm{\acute{e}t}}(X,K)\otimes_{\mathbb Z}\mathbb Q
 \xrightarrow{\simeq}
 R\Gamma_{\mathrm{\acute{e}t}}(X,K\otimes_{\mathbb Z}\mathbb Q).
\end{equation}
It holds for every complex of \'etale sheaves $K$ on every noetherian finite-dimensional $X$, by \cite[Proposition~1.1.11]{CDEt}. This result does not require a bound on the torsion \'etale cohomological dimensions of the residue fields.

\begin{lemma}[Rationalization of the right adjoint]\label{lem:forget-rationalization}
Let $S$ be noetherian of finite dimension. For every $M\in\mathrm{DM}_{\mathrm{\acute{e}t}}(S,\mathbb Z)$, the natural map
\begin{equation}\label{eq:forget-rationalization}
 G_{S,\mathbb Z}(M)\otimes_{\mathbb Z}\mathbb Q
 \longrightarrow
 G_{S,\mathbb Z}(M\otimes_{\mathbb Z}\mathbb Q)
\end{equation}
is an equivalence.
\end{lemma}
\begin{proof}
On \'etale sheaves, forgetting transfers is exact and preserves colimits \cite[Proposition~2.1.10(3)]{CDEt}. Thus it commutes with rationalization on complexes. Its derived functor preserves $\mathbb A^1$-equivalences \cite[Proposition~2.2.8]{CDEt}, and sends $\mathbb A^1$-local objects to $\mathbb A^1$-local objects by the defining adjunction. It follows that the effective forgetful right adjoint commutes with rationalization. We explain the additional point required by Tate stabilization.

Represent $M$ by a Tate $\Omega$-spectrum whose levels $K_n$ are \'etale-local and $\mathbb A^1$-local complexes with transfers. Write $T_{\mathbb Z}^{\mathrm{tr}}=F_{S,\mathbb Z}^{\mathrm{eff}}(T_{\mathbb Z})$ for the transferred Tate suspension object. Derived sections of such complexes are computed after forgetting transfers; see \cite[Proposition~2.2.3]{CDEt}. For a smooth $S$-scheme $X$, the schemes $X$, $X\times\mathbb A^1$, and $X\times\mathbb G_m$ are noetherian of finite dimension. Applying \eqref{eq:etale-rational-sections} to these schemes shows that rationalization preserves the $\mathbb A^1$-locality conditions on the levels.

The Tate $\Omega$-spectrum condition is likewise preserved. To see this explicitly, maps out of a Tate suspension of a smooth representable are computed from derived sections on $X\times\mathbb G_m$ and on $X$, using the reduced cofiber at the unit section and the shift by one. Formula~\eqref{eq:etale-rational-sections}, together with exactness of tensoring with $\mathbb Q$, commutes with these finite fibres and shifts. Consequently the structural equivalences
\[
 K_n\simeq\mathbf R\underline{\operatorname{Hom}}(T_{\mathbb Z}^{\mathrm{tr}},K_{n+1})
\]
remain equivalences after levelwise rationalization. The same argument applies to spectra without transfers.

Levelwise rationalization therefore already gives a local Tate $\Omega$-spectrum, and computes rationalization in the stable motivic category. The right adjoint on these spectra is levelwise forgetting of transfers: its compatibility with Tate loops follows by adjunction from the compatibility of the left adjoint with the Tate suspension object. Since forgetting transfers on complexes commutes with rationalization, applying it levelwise gives \eqref{eq:forget-rationalization}. These identifications are induced by the natural projection maps, so the resulting equivalence is the canonical map in the statement.
\end{proof}

\subsection{Torsion coefficients and the open where they are invertible}
We include the reduction to the prime-to-characteristic locus to make clear that no localization theorem for integral motives with finite transfers is being assumed.

\begin{lemma}\label{lem:remove-characteristic-fibre}
Let $p$ be a prime and $j:U=S[1/p]\hookrightarrow S$. For $\mathcal M=\mathrm{DA}_{\mathrm{\acute{e}t}}$ and for $\mathcal M=\mathrm{DM}_{\mathrm{\acute{e}t}}$, restriction induces an equivalence
\[
 j^*:\mathcal M(S,\mathbb F_p)\xrightarrow{\simeq}
 \mathcal M(U,\mathbb F_p).
\]
These equivalences are compatible with $F_{S,\mathbb F_p}$.
\end{lemma}
\begin{proof}
Let $i:Z\hookrightarrow S$ be the closed fibre defined by $p$. On a scheme of characteristic $p$, the Artin--Schreier sequence is an exact sequence of \'etale sheaves of $\mathbb F_p$-vector spaces:
\[
 0\longrightarrow\mathbb F_p\longrightarrow\mathbb G_a
 \xrightarrow{\,x\mapsto x^p-x\,}\mathbb G_a\longrightarrow0.
\]
Multiplication by the affine-line parameter contracts the additive sheaf $\mathbb G_a$ in the effective $\mathbb A^1$-derived category. The sequence therefore makes the unit zero in $\mathrm{DA}_{\mathrm{\acute{e}t}}(Z,\mathbb F_p)$, so this category vanishes.

Localization for motives without transfers now gives
\begin{equation}\label{eq:torsion-open-unit}
 j_\sharp\mathbf1_U\xrightarrow{\simeq}\mathbf1_S
 \quad\text{in }\mathrm{DA}_{\mathrm{\acute{e}t}}(S,\mathbb F_p).
\end{equation}
The functor adding transfers is symmetric monoidal and commutes with smooth pushforward, hence with $j_\sharp$. Applying it proves the same unit equivalence in $\mathrm{DM}_{\mathrm{\acute{e}t}}(S,\mathbb F_p)$. For either construction, the smooth projection formula yields
\[
 j_\sharp j^*M\simeq M\otimes j_\sharp\mathbf1_U\simeq M.
\]
The other composite $j^*j_\sharp$ is the identity for an open immersion. Thus $j^*$ and $j_\sharp$ are inverse equivalences in both theories. Every map used is compatible with adding transfers. This is also the argument of \cite[Proposition~A.3.4]{CDEt}.
\end{proof}

\begin{proposition}[Comparison with finite-field coefficients]\label{prop:torsion-comparison}
For every prime $p$ and every noetherian finite-dimensional $S$, the canonical comparison
\[
 F_{S,\mathbb F_p}:\mathrm{DA}_{\mathrm{\acute{e}t}}(S,\mathbb F_p)
 \longrightarrow\mathrm{DM}_{\mathrm{\acute{e}t}}(S,\mathbb F_p)
\]
is an equivalence.
\end{proposition}
\begin{proof}
By Lemma~\ref{lem:remove-characteristic-fibre}, it suffices to prove the assertion over $U=S[1/p]$. Consider the canonical composable functors
\begin{equation}\label{eq:rigidity-triangle}
 D(U_{\mathrm{\acute{e}t}},\mathbb F_p)
 \xrightarrow{\ a_U\ }
 \mathrm{DA}_{\mathrm{\acute{e}t}}(U,\mathbb F_p)
 \xrightarrow{\ F_{U,\mathbb F_p}\ }
 \mathrm{DM}_{\mathrm{\acute{e}t}}(U,\mathbb F_p).
\end{equation}
The functor $a_U$ is an equivalence by \cite[Corollary~3.2]{BachmannEtale}. The composite is the small-site realization functor with transfers, which is an equivalence by \cite[Theorem~4.5.2]{CDEt}. The identification of the composite is canonical: for an \'etale $U$-scheme $V$, its graph map from the ordinary represented \'etale sheaf to the represented sheaf with transfers is an isomorphism \cite[Corollary~2.1.9]{CDEt}; the construction then extends by the derived adjunctions and localization. Thus two-out-of-three applied to \eqref{eq:rigidity-triangle} proves the assertion for the specified comparison functor. Neither rigidity theorem excludes $p=2$ in this formulation.
\end{proof}

\subsection{Reduction of the main theorem to rational coefficients}
\begin{proposition}\label{prop:integral-reduction}
Fix a noetherian scheme $S$ of finite dimension. If $F_{S,\mathbb Q}$ is an equivalence, then $F_{S,R}$ is an equivalence for every commutative ring $R$.
\end{proposition}
\begin{proof}
Apply Lemma~\ref{lem:arithmetic-detection} to $F_{S,\mathbb Z}\dashv G_{S,\mathbb Z}$. Lemma~\ref{lem:forget-rationalization} supplies its rational projection formula. Proposition~\ref{prop:torsion-comparison} supplies the comparisons over every $\mathbb F_p$, and the rational comparison is the hypothesis. The lemma proves that $F_{S,\mathbb Z}$ is an equivalence. Lemma~\ref{lem:coefficient-extension} then proves the assertion for every $R$.
\end{proof}

Thus the remaining task is the rational comparison. All motives and all geometric constructions in Sections~\ref{sec:formal}--\ref{sec:conclusion}, until the final proof of Theorem~\ref{thm:main}, have rational coefficients.

\section{Rational motives and the split-torus reduction}\label{sec:formal}
\subsection{Spaces, spectra, and the two topologies}
Let $\mathrm{Sm}_S$ be the category of separated smooth $S$-schemes of finite type. We work with Nisnevich sheaves of spaces on this category and their $\mathbb A^1$-localization $\mathcal H(S)$. Equivalently, one may use Nisnevich-local simplicial presheaves. The site is locally of finite cohomological dimension, and the usual stalkwise and hypercomplete presentations agree. Pointing and stabilizing with respect to
\[
 T=S^{2,1}=S^1\wedge(\mathbb G_m,1)
\]
gives $\mathrm{SH}(S)$.

The rational $\mathbb A^1$-derived category $\mathrm{DA}(S,\mathbb Q)$ and its \'etale counterpart $\mathrm{DA}_{\mathrm{\acute{e}t}}(S,\mathbb Q)$ are obtained by rational linearization, $\mathbb A^1$-localization, and Tate stabilization, using the Nisnevich and \'etale topologies respectively. The category $\mathrm{DM}(S,\mathbb Q)$ uses Nisnevich sheaves with finite transfers. All categories and functors are derived. We use their stable $\infty$-categorical enhancements; the quoted model-categorical comparison results induce the corresponding equivalences of enhancements.

Write
\[
 M_S:\operatorname{Shv}_{\mathrm{Nis}}(\mathrm{Sm}_S)
 \longrightarrow \mathrm{DA}_{\mathrm{\acute{e}t}}(S,\mathbb Q)
\]
for the motive functor, with a disjoint base point understood. It preserves colimits and satisfies
\begin{equation}\label{eq:monoidal-motive}
 M_S(F\times G)\simeq M_S(F)\otimes M_S(G).
\end{equation}
For a pointed space $F$, put
$\widetilde M_S(F)=\operatorname{cofib}(\mathbf1_S\to M_S(F))$.

For every separated finite-type $S$-scheme $W$, not necessarily smooth, write
\begin{equation}\label{eq:restricted-representable}
 h_W(U)=\operatorname{Hom}_S(U,W),\qquad U\in\mathrm{Sm}_S.
\end{equation}
The expression $M_S(h_W)$ always refers to this restricted representable. We do not implicitly enlarge the site.

\subsection{Rational linearization and the plus part}
For a space $F$, let $\mathbb Q[F]$ be its derived free rational linearization. For a pointed space, use the reduced linearization $\widetilde{\mathbb Q}[F]$. Applied levelwise to a $T$-spectrum $E$, this gives the free rational $T$-spectrum $\mathbb Q[E]$.

\begin{proposition}\label{prop:rational-sphere}
There are canonical symmetric monoidal equivalences
\[
 \mathrm{DA}(S,\mathbb Q)\simeq\mathrm{SH}(S)_{\mathbb Q},
 \qquad
 \mathrm{DA}_{\mathrm{\acute{e}t}}(S,\mathbb Q)
 \simeq\mathrm{SH}(S)_{\mathbb Q,+}.
\]
Under the first equivalence, $\mathbb Q[\mathbb S_S]$ is the rational sphere. Under the second, its plus part is the unit. In particular,
\[
 \operatorname{Mod}_{\mathbb Q[\mathbb S_S]_+}(\mathrm{SH}(S))
 \simeq\mathrm{DA}_{\mathrm{\acute{e}t}}(S,\mathbb Q).
\]
\end{proposition}
\begin{proof}
The first assertion is rational stable linearization
\cite[(5.3.35.2)]{CD}. The second follows from the identification of Beilinson motives with the rational plus part and with the rational \'etale stable $\mathbb A^1$-derived category
\cite[Theorems~16.2.13 and~16.2.18]{CD}. These equivalences preserve units. The statement about modules follows because rational plus localization is a smashing localization.
\end{proof}

\subsection{The transferred unit}
For $X\in\mathrm{Sm}_S$, write $L_S(X)_{\mathbb Q}$ for the rational sheaf represented by $X$ in finite correspondences, with transfers forgotten. Its values are finite relative zero-cycles in the sense of Suslin--Voevodsky \cite{SV,CD}. Extend $L_S$ by derived Kan extension to spaces, taking values in connective complexes of sheaves with transfers. For a pointed space $F$, set
\[
 L_{S,*}(F)_{\mathbb Q}
 =\operatorname{cofib}\bigl(L_S(*)_{\mathbb Q}\to L_S(F)_{\mathbb Q}\bigr).
\]
In this formula the cofiber is formed in complexes before taking underlying spaces by Dold--Kan.

The transferred unit has the level presentation
\begin{equation}\label{eq:HQ-levels}
 (H\mathbb Q_S)^n=L_{S,*}(S^{2n,n})_{\mathbb Q}.
\end{equation}
This is a presentation of $\varphi_{S*}(\mathbf1_S)$, not a definition by underived forgetting of a fibrant replacement. Indeed, forgetting transfers on complexes of Tate spectra preserves stable $\mathbb A^1$-equivalences by \cite[Proposition~6.3.9]{CD}, whose hypotheses are verified for sheaves with transfers in \cite[\S11.4.1]{CD}. Rational stable linearization then gives \eqref{eq:HQ-levels}.

Graphs give a canonical ring map
\[
 \mathbb Q[\mathbb S_S]\longrightarrow H\mathbb Q_S.
\]
The orientation of the transferred unit makes its forgotten spectrum a Beilinson motive \cite[\S16.1.1]{CD}. It therefore lies in the rational plus part, and the map factors through
\begin{equation}\label{eq:Phi}
 \Phi_S:\mathbb Q[\mathbb S_S]_+\longrightarrow H\mathbb Q_S.
\end{equation}

\begin{remark}\label{rem:underlying-space}
When $A$ is an abelian sheaf, $M_S(A)$ means the motive of its underlying space. It is generally different from the object obtained by localizing the complex $A[0]$. In particular, multiplication by an integer on $A$ is an endomorphism of a space; its image under $M_S$ must not be confused with scalar multiplication in a rational stable category.
\end{remark}

\subsection{Delooping and the graph criterion}
\begin{lemma}\label{lem:deloop}
Let $f:A\to B$ be a morphism of sheaves of abelian groups. If $M_S(f)$ is an equivalence, then the induced map on the underlying spaces of $A[n]\to B[n]$ is a motivic equivalence in $\mathrm{DA}_{\mathrm{\acute{e}t}}(S,\mathbb Q)$ for every $n\geq0$.
\end{lemma}
\begin{proof}
The classifying space $BA$ is the geometric realization of the simplicial space with $q$-simplices $A^q$. By \eqref{eq:monoidal-motive}, $M_S$ sends the map in each simplicial degree to an equivalence; it also preserves the realization. Thus $M_S(BA)\to M_S(BB)$ is an equivalence. Iterate this argument. The $n$-fold classifying space is the underlying space of the complex $A[n]$.
\end{proof}

\begin{proposition}[Split-torus criterion]\label{prop:torus-criterion}
Suppose that, for every $r\geq0$, the canonical graph map
\begin{equation}\label{eq:graph-criterion}
 M_S(\mathbb Q[\mathbb G_{m,S}^{r}])
 \longrightarrow M_S(L_S(\mathbb G_{m,S}^{r})_{\mathbb Q})
\end{equation}
is an equivalence. Then $\Phi_S$ is an equivalence of commutative ring spectra.
\end{proposition}
\begin{proof}
For $1\leq j\leq r$, let $e_j$ replace the $j$th coordinate of $\mathbb G_m^r$ by $1$. These maps give commuting idempotents on both abelian sheaves in \eqref{eq:graph-criterion}, and the graph map commutes with them. The images of
\[
 \prod_{j=1}^{r}(1-e_j)
\]
are $\widetilde{\mathbb Q}[\mathbb G_m^{\wedge r}]$ and
$L_{S,*}(\mathbb G_m^{\wedge r})_{\mathbb Q}$. This follows by taking the total cofiber of the split cube of coordinate faces. Their comparison is a retract of the graph comparison already as a map of underlying spaces. Applying $M_S$ preserves this retract.

By Lemma~\ref{lem:deloop} and the identities
\[
 \widetilde{\mathbb Q}[S^{2r,r}]
 \simeq\widetilde{\mathbb Q}[\mathbb G_m^{\wedge r}][r],
 \qquad
 L_{S,*}(S^{2r,r})_{\mathbb Q}
 \simeq L_{S,*}(\mathbb G_m^{\wedge r})_{\mathbb Q}[r],
\]
the reduced motives of the level spaces of the two spectra agree.

For any $T$-spectrum $E$, its stable object is the colimit of its level filtration:
\[
 E\simeq\operatorname*{colim}_{r}
 \Sigma_T^{\infty}E^r(-r)[-2r].
\]
Proposition~\ref{prop:rational-sphere} therefore implies that the graph map of spectra is an equivalence after smashing with the rational plus sphere. Its source then becomes $\mathbb Q[\mathbb S_S]_+$ and its target remains $H\mathbb Q_S$. The resulting map is \eqref{eq:Phi}. For $r=0$, the required comparison is the identity of the constant rational sheaf.
\end{proof}

\section{Effective cycles and power maps}\label{sec:cycles}
\subsection{Division of multiplicities}
Let $A$ be a sheaf of commutative monoids. Define its divisible localization by
\begin{equation}\label{eq:divisible}
 A^{\mathrm{div}}=
 \operatorname*{colim}\bigl(
 A\xrightarrow{\times2}A\xrightarrow{\times3}A
 \xrightarrow{\times4}\cdots\bigr).
\end{equation}
This is localization of the monoid at all positive integers. For a cancellative torsion-free monoid, it is the submonoid of its rationalized group completion consisting of elements $a/N$. The colimit in \eqref{eq:divisible} is taken in sheaves.

For a smooth affine $S$-scheme $X$, set
\begin{equation}\label{eq:PD}
 P_X=\coprod_{n\geq0}a_{\mathrm{Nis}}(h_{X^n}/\Sigma_n),
 \qquad
 D_X=\coprod_{n\geq0}h_{\operatorname{Sym}_S^nX}.
\end{equation}
The quotients in $P_X$ are ordinary orbit sheaves, not homotopy quotients. The first monoid is free on $h_X$; the second uses geometric symmetric powers. The coproducts allow locally constant degrees on disconnected tests.

Let $\mathcal C_S(X)$ denote the sheaf of effective rational finite relative cycles on $X$. The support of a section over $U$ is finite over $U$, and ``relative cycle'' has the Suslin--Voevodsky meaning.

\begin{proposition}[Cycle presentation]\label{prop:cycle-presentation}
Let $S$ be seminormal and noetherian, and let $X/S$ be smooth and affine. There is a natural isomorphism of sheaves of commutative monoids
\[
 D_X^{\mathrm{div}}\simeq\mathcal C_S(X).
\]
It respects addition, base change between test schemes, and pushforward along finite endomorphisms of $X$.
\end{proposition}
\begin{proof}
Every smooth test scheme $U/S$ is seminormal: seminormality ascends under regular morphisms of noetherian schemes \cite[Proposition~38]{Kollar}. Work on affine tests, so that the number of irreducible components is finite.

Rydh's comparison theorem \cite[Theorem~10.17]{RydhCycles} identifies families of cycles over $U$ with effective cycles that satisfy the limit-cycle condition and whose specialization at every point has quasi-integral coefficients. In relative dimension zero, families with finite support and degree $n$ are represented by $\Gamma_U^n(X_U)$. Since $X_U/U$ is flat, the canonical map
\[
 \operatorname{Sym}_U^n(X_U)\longrightarrow\Gamma_U^n(X_U)
\]
is an isomorphism \cite[Corollary~4.2.5]{RydhZero}.

An effective rational relative cycle satisfies the limit-cycle condition. By \cite[Corollary~8.1.54]{CD}, some positive integral multiple of it is universally integral: all its base changes have integral coefficients. In particular, all its point specializations are quasi-integral. Rydh's theorem therefore realizes this multiple by a family of zero-cycles, hence by a section of a geometric symmetric power. This proves surjectivity after \eqref{eq:divisible}.

Conversely, such a family has a well-defined rational relative cycle. Two families over the reduced scheme $U$ with the same generic cycles coincide, again by the cited comparison theorem. Thus the map is injective after division of multiplicities. The constructions are compatible with the operations asserted in the proposition and glue on the smooth site.
\end{proof}

\begin{remark}\label{rem:cycle-characteristic}
In positive or mixed characteristic, integral generic coefficients alone do not suffice in the preceding proof: the quasi-integrality condition must hold on every fibre. This is why we use universal integrality and Theorem~10.17 of \cite{RydhCycles}, rather than its characteristic-zero Corollary~10.18. Division of multiplicities in \eqref{eq:divisible} is distinct from rational coefficients on the motive of a space.
\end{remark}

\subsection{Group completion}
\begin{lemma}\label{lem:cycle-completion}
For a smooth affine $S$-scheme $X$, the canonical map
\[
 a_{\mathrm{Nis}}\bigl(\mathcal C_S(X)^{\mathrm{gp}}\bigr)
 \longrightarrow L_S(X)_{\mathbb Q}
\]
is an isomorphism.
\end{lemma}
\begin{proof}
Injectivity follows from cancellation. For surjectivity, work at a Nisnevich stalk $U=\operatorname{Spec}R$, where $R$ is henselian local. Let $\alpha$ be a finite rational relative cycle on $X_U$, and let $W$ be the reduced union of its negatively weighted components. Its support is finite over $U$.

We construct a finite flat envelope of $W$. Work on a relative-dimension stratum of $X_U$, of dimension $d$. In the ideal of $W$, choose $d$ functions that form a system of parameters in the smooth closed fibre at each of the finitely many closed-fibre points of $W$. Prime avoidance in the corresponding semilocal ring gives such functions, over finite as well as infinite residue fields. Clearing the finitely many denominators gives functions in the ideal of $W$ on $X_U$ without changing their parameter property at the selected points. Their images form regular sequences in the regular local rings of that fibre. The local criterion for flatness shows that their common zero scheme is flat and quasi-finite over $U$ in a neighbourhood of these points. This neighbourhood contains all of $W$: its intersection with the omitted closed set is finite over $U$ and has empty closed fibre. A finite scheme over a local base with empty closed fibre is empty, which proves the assertion about the neighbourhood.

A separated quasi-finite scheme over a henselian local scheme has an open and closed finite part containing its closed fibre. This part contains $W$ and is flat over $U$. Taking the union over the finitely many dimension strata gives a finite locally free envelope $V$. Its fundamental relative cycle $\beta$ is effective and contains every negatively weighted generic component of $\alpha$ with positive multiplicity. Thus $\alpha+N\beta$ is effective for a sufficiently large integer $N$, and
\[
 \alpha=(\alpha+N\beta)-N\beta.
\]
This proves surjectivity at all stalks. It is the affine relative version of the argument in \cite[Lemma~3.8]{VoevEM}.
\end{proof}

\begin{lemma}[Motivic group completion]\label{lem:group-completion}
Let $A\to B$ be a morphism of sheaves of cancellative commutative monoids. If $M_S(A)\to M_S(B)$ is an equivalence, then so is
\[
 M_S(a_{\mathrm{Nis}}A^{\mathrm{gp}})
 \longrightarrow M_S(a_{\mathrm{Nis}}B^{\mathrm{gp}}).
\]
The assertion also holds after tensoring the group completions with $\mathbb Q$.
\end{lemma}
\begin{proof}
For a cancellative monoid $D$, form the category with objects $(a,b)\in D^2$ and arrows
\[
 (a,b)\longrightarrow(a+c,b+c),\qquad c\in D.
\]
Its nerve has $D^{q+2}$ in simplicial degree $q$. Components are indexed by $D^{\mathrm{gp}}$ through $(a,b)\mapsto a-b$. Each component is filtered: if $a-b=a'-b'$, then $a+b'=a'+b$, so translation by $b'$ and by $b$ gives a common target. Cancellation makes parallel arrows unique. The components are therefore contractible.

Applying this construction at sheaf stalks identifies $a_{\mathrm{Nis}}D^{\mathrm{gp}}$ with the realization of the indicated nerve. Equation~\eqref{eq:monoidal-motive} and preservation of realizations prove the assertion. Rationalization of the group completion is a filtered colimit along multiplication by $2,3,4,\ldots$; apply preservation of filtered colimits for the last statement.
\end{proof}

\subsection{Power maps and a homotopy of Kummer cycles}
In this subsection, let $X=\mathbb G_{m,S}^{r}$. For $N\geq1$, denote by
\[
 u_N:X\longrightarrow X,\qquad
 (t_1,\ldots,t_r)\longmapsto(t_1^N,\ldots,t_r^N)
\]
the coordinate power map. It is finite locally free of degree $N^r$.

\begin{lemma}[Power maps on cycle spaces]\label{lem:kummer}
For every noetherian base $S$, the morphism
\[
 u_{N*}:\mathcal C_S(X)\longrightarrow\mathcal C_S(X)
\]
is an $\mathbb A^1$-homotopy equivalence of underlying spaces. A homotopy inverse is $N^{-r}u_N^*$.
\end{lemma}
\begin{proof}
For $r=0$, the map is the identity. Assume $r>0$. Finite-flat pullback and finite pushforward of relative cycles give
\begin{equation}\label{eq:push-pull}
 u_{N*}u_N^*=N^r\operatorname{id}.
\end{equation}
Let $\kappa_N$ be the fundamental cycle of the finite locally free kernel $\mu_N^r\subset X$. The identity
\[
 X\times_{u_N,X,u_N}X\simeq\mu_N^r\times_S X
\]
and finite-flat base change give
\begin{equation}\label{eq:pull-push}
 u_N^*u_{N*}(\alpha)=\kappa_N*\alpha,
\end{equation}
where $*$ denotes convolution under multiplication on $X$. These are identities of relative cycles, with scheme-theoretic multiplicities retained; see the pullback, pushforward, and base-change rules in \cite[\S8.2]{CD}.

Consider the polynomial
\begin{equation}\label{eq:kummer-polynomial}
 f_\lambda(T)=(1-\lambda)(T^N-1)^2+\lambda(T-1)^{2N}.
\end{equation}
The squares make the two endpoint polynomials monic of the same degree and with the same constant term over every base. Thus $f_\lambda$ is monic of degree $2N$, with constant term $1$. It therefore defines a finite locally free family inside $\mathbb G_m\times\mathbb A^1_S$. At $\lambda=0$ and $\lambda=1$, its cycles are $2[\mu_N]$ and $2N[1]$, respectively. Take the external product of $r$ copies over the same parameter line and divide the cycle by $2^r$. This gives a family of rational effective cycles from $\kappa_N$ to $N^r[1_X]$. The family is a path in the space of cycles: it retains the multiplicities of the roots-of-unity group scheme, including when $N$ is not invertible on $S$.

Convolution with a variable cycle, followed by division by $N^r$, gives a homotopy from $N^{-r}u_N^*u_{N*}$ to the identity. Together with \eqref{eq:push-pull}, this proves the lemma. The families and convolution commute with base change on test schemes, so the homotopy is one of sheaves of spaces, not just of their values on fields.
\end{proof}

\begin{lemma}\label{lem:power-free}
The endomorphism induced by $u_N$ on $M_S(P_X^{\mathrm{div}})$ is an equivalence. If $S$ is seminormal, the same is true on $M_S(D_X^{\mathrm{div}})$.
\end{lemma}
\begin{proof}
Put $L=\mathbf1_S(1)[1]$. The pointed splitting at $1$ gives
\[
 M_S(X)\simeq
 \bigoplus_{J\subseteq\{1,\ldots,r\}}L^{\otimes|J|}.
\]
On the summand indexed by $J$, $u_N$ acts by $N^{|J|}$. The calculation on $\mathbb G_m$ is the standard degree calculation for the Tate class; it may be checked over the regular base $\operatorname{Spec}\mathbb Z$ using the established comparison \cite[Corollary~16.2.22]{CD} and then pulled back to $S$. Thus it does not assume our theorem for $S$.

The resulting equivalence passes through tensor products and finite-group homotopy orbits. The map from the homotopy quotient of a discrete set to its orbit set is a rational homology equivalence, because its components have finite stabilizers. Applying rational linearization and sheafification shows that
\[
 M_S(X^n)_{h\Sigma_n}
 \simeq M_S\bigl(a_{\mathrm{Nis}}(h_{X^n}/\Sigma_n)\bigr).
\]
Coproducts and the filtered colimit \eqref{eq:divisible} give the first assertion. Proposition~\ref{prop:cycle-presentation} and Lemma~\ref{lem:kummer} give the second.
\end{proof}

\section{Quotient errors and excision}\label{sec:quotients}
Let $G$ be a finite constant group acting on a smooth quasi-projective $S$-scheme $Y$. Geometric quotients exist as schemes in this setting. Define
\begin{equation}\label{eq:error}
 \mathcal E_S(Y,G)=\operatorname{cofib}\left(
 M_S(Y)_{hG}\longrightarrow M_S(h_{Y/G})\right).
\end{equation}
This is functorial for equivariant morphisms, including morphisms that increase stabilizers. Homotopy orbits and cofibers are taken in the stable category $\mathrm{DA}_{\mathrm{\acute{e}t}}(S,\mathbb Q)$.

\begin{lemma}\label{lem:error-excision}
The functor $\mathcal E_S(-,G)$ has the following properties.
\begin{enumerate}[label=\textup{(\roman*)},leftmargin=2em]
\item It satisfies Mayer--Vietoris for invariant open covers and is invariant under product with $\mathbb A^1$ with trivial $G$-action.
\item If $V\to Y$ is equivariant \'etale, preserves geometric stabilizers, and is an isomorphism over an invariant closed subscheme $Z\subset Y$, then the square of errors for $V$, $Y$, and their complements of $Z$ is a pushout.
\item If $G$ acts freely, then $\mathcal E_S(Y,G)=0$.
\item For $H\subseteq G$ and an $H$-scheme $V$, induction gives
\[
 \mathcal E_S(G\times_H V,G)\simeq\mathcal E_S(V,H).
\]
\end{enumerate}
\end{lemma}
\begin{proof}
An invariant open has an open geometric quotient. Consequently an invariant open cover induces an open cover of the quotient, proving Mayer--Vietoris for both terms of \eqref{eq:error}. Also,
\[
 (Y\times\mathbb A^1)/G\simeq(Y/G)\times\mathbb A^1,
\]
since the group acts trivially on the parameter. This proves (i).

For (ii), the fixed-point-reflecting descent theorem for finite quotients gives a cartesian square
\[
\begin{tikzcd}
 V\arrow[r]\arrow[d]&Y\arrow[d]\\
 V/G\arrow[r]&Y/G
\end{tikzcd}
\]
with \'etale lower map \cite[Proposition~3.5, Definition~3.6 and Corollary~5.4]{RydhQuotients}. Let $\overline Z$ be the reduced closed image of $Z$ in $Y/G$. The lower map is an isomorphism over $\overline Z$. Indeed, its geometric fibre at a point of $\overline Z$ is identified, after lifting the point to $Z$, with the singleton fibre of $V\to Y$. An \'etale morphism with singleton geometric fibres everywhere is an isomorphism.

The quotient square with complementary opens is therefore Nisnevich distinguished. This gives a pushout for restricted representable sheaves: pullback along a smooth test scheme mapping to $Y/G$ produces a Nisnevich distinguished square of smooth $S$-schemes, whose sheaf pushout is the test scheme. The assertion can thus be checked locally on every test. No identification of $Z/G$ with its scheme-theoretic image is required. Taking motives and comparing with the ordinary \'etale square upstairs proves (ii).

For (iii), $Y\to Y/G$ is a finite \'etale $G$-torsor, and its \v Cech nerve is the action groupoid. The quotient is smooth, and \'etale descent gives the comparison. Finally, the quotient of $G\times_H V$ by $G$ is $V/H$, and transitivity of homotopy orbits identifies the other terms, proving (iv).
\end{proof}

\subsection{Maximal-stabilizer neighbourhoods}
We record the permutation geometry used in both the tame and arithmetic arguments. Let
\[
 Y\subset(\mathbb G_m^r)^n\times\mathbb A^d
\]
be an invariant open subscheme for $G\subseteq\Sigma_n$, with trivial action on the affine coordinates. Let $m$ be the largest order of a geometric stabilizer on $Y$.

For every subgroup $H\subseteq G$ of order $m$, the fixed locus $Y^H$ is smooth: it is an open in the coordinate diagonal that identifies the torus factors in each $H$-orbit. Distinct nonempty $Y^H$ of this order are disjoint; an intersection would have a geometric point with larger stabilizer. Hence
\begin{equation}\label{eq:maximal-fixed}
 Z=\coprod_{|H|=m}Y^H
\end{equation}
is a smooth invariant closed subscheme, and $Y\setminus Z$ has smaller stabilizers.

For a conjugacy representative $H$, let $N=N_G(H)$. Remove from $Y$ every other order-$m$ fixed locus and the fixed loci of every element of $G\setminus N$. The resulting $N$-invariant open $V_H$ contains $Y^H$, and all stabilizers on it are contained in $N$. The map
\begin{equation}\label{eq:normalizer-neighbourhood}
 \coprod_{[H]}G\times_{N_G(H)}V_H\longrightarrow Y
\end{equation}
is \'etale, preserves geometric stabilizers, and is an isomorphism over \eqref{eq:maximal-fixed}. It is a finite disjoint union of open immersions on the underlying schemes, followed on each component by a coordinate permutation.

Thus induction and Lemma~\ref{lem:error-excision} reduce a vanishing assertion for the error to the case
\begin{equation}\label{eq:normalizer-case}
 H\triangleleft N,\qquad Z=V^H,
\end{equation}
where $Z$ is the maximal-stabilizer locus of $V$. On $Z$, $H$ acts trivially and $N/H$ acts freely. Rational homology and \'etale descent give
\begin{equation}\label{eq:fixed-error-zero}
 \mathcal E_S(Z,N)=0.
\end{equation}
Indeed, $M_S(Z)_{hH}\simeq M_S(Z)$, compatibly with the residual action, and $Z/N\simeq Z/(N/H)$.

\begin{proposition}[Tame permutation comparison]\label{prop:tame-permutation}
Suppose $n!$ is invertible on $S$. For any invariant open $Y\subset(\mathbb G_{m,S}^r)^n\times_S\mathbb A_S^d$ as above,
\[
 \mathcal E_S(Y,G)=0.
\]
\end{proposition}
\begin{proof}
Induct on the maximum stabilizer order. The free case follows from Lemma~\ref{lem:error-excision}. Use \eqref{eq:normalizer-neighbourhood} for the induction step. In the case \eqref{eq:normalizer-case}, replace each torus coordinate in an $H$-orbit $I$ by its average
\[
 R_{I,j}=|I|^{-1}\sum_{a\in I}x_{a,j}.
\]
These functions define an $N$-equivariant retraction to $Z$ on an invariant open neighbourhood $V'\subset V$ of all of $Z$. The complement away from $Z$ has smaller stabilizers, so Mayer--Vietoris identifies the errors of $V'$ and $V$.

Interpolate from $x_{a,j}$ to $R_{I,j}$, keeping the auxiliary affine coordinates fixed. The homotopy is defined as a map to $V$ on an invariant open $W\subset V'\times\mathbb A^1$ containing both endpoint fibres and $(V'\cap Z)\times\mathbb A^1$. The complementary part lies off the maximal fixed locus. Induction and Mayer--Vietoris therefore identify $\mathcal E_S(W,N)$ with $\mathcal E_S(V',N)$. The two endpoint sections are inverse to the same projection equivalence, so they agree after applying the error functor. Their composites with the homotopy identify the inclusion $V'\to V$ with a map through $Z$. Equation~\eqref{eq:fixed-error-zero} makes this map zero. Since the inclusion induces an equivalence on errors, those errors vanish.
\end{proof}

\section{Arithmetic neighbourhoods and permutation quotients}\label{sec:arithmetic}
Fix a prime $p$ and a noetherian scheme $S$. Set
\[
 S_p=S\times_{\operatorname{Spec}\mathbb Z}\operatorname{Spec}\mathbb F_p,
 \qquad i_p:S_p\longrightarrow S.
\]
The argument in this section concerns $i_p^*\mathcal E_S(Y,G)$. We never identify it with $\mathcal E_{S_p}(Y_{S_p},G)$.

Our models are invariant opens
\[
 Y_0\subset(\mathbb G_{m,\mathbb Z}^r)^n\times_{\mathbb Z}\mathbb A_{\mathbb Z}^d
\]
with permutation actions. We also allow coordinate fixed loci, invariant opens, finite disjoint unions, and the induced schemes in \eqref{eq:normalizer-neighbourhood}. A coordinate fixed locus of the ambient torus is again a split torus, with one block of $r$ coordinates for each orbit; the fixed locus of a model is an invariant open in that torus times $\mathbb A^d$. Its normalizer acts by permuting the orbit blocks. If the resulting action has a constant kernel, rational homotopy orbits allow that kernel to be removed. The constructions below are performed componentwise on disjoint unions. Put
\[
 E_p(Y_0,G)=i_p^*\mathcal E_S(Y_0\times_{\mathbb Z}S,G).
\]
An open $V_0\subset Y_0$ containing the entire fibre $(Y_0)_{\mathbb F_p}$ will be called a \emph{$p$-neighbourhood}. Such opens are always invariant when used with group actions.

\subsection{Germs along the characteristic-$p$ fibre}
\begin{lemma}\label{lem:p-neighbourhood}
An inclusion of a $p$-neighbourhood induces an equivalence
\[
 E_p(V_0,G)\xrightarrow{\simeq}E_p(Y_0,G).
\]
The same assertion holds with auxiliary affine parameters.
\end{lemma}
\begin{proof}
We have the invariant open cover
\[
 Y_0=V_0\cup Y_0[1/p].
\]
For an $S[1/p]$-scheme $W$, even a singular one, extension by zero on the smooth site gives
\[
 M_S(h_W)\simeq j_\sharp M_{S[1/p]}(h_W),
 \qquad j:S[1/p]\hookrightarrow S.
\]
Its restriction by $i_p^*$ is zero, since $i_p^*j_\sharp=0$ in the motivic category $\mathrm{DA}_{\mathrm{\acute{e}t}}$. Apply this to both terms of the error on $Y_0[1/p]$ and on its intersection with $V_0$. Mayer--Vietoris proves the assertion. The proof is unchanged after adjoining affine parameters.
\end{proof}

It is useful to describe a morphism on a germ as a map $V_0\to Z_0$ defined on a $p$-neighbourhood of $Y_0$; two representatives are identified after restriction to a common $p$-neighbourhood. Intersections provide common refinements, and composition is defined by inverse image of the next domain. Lemma~\ref{lem:p-neighbourhood} allows us to apply $E_p$ to these germs. Equivalently, one uses the localization at these open immersions. Inverses to their images are taken in the stable $\infty$-category.

Coordinatewise $p$th powers on the ambient model define a germ
\[
 P_{Y_0}:Y_0\dashrightarrow Y_0.
\]
Its domain is $Y_0\cap P^{-1}(Y_0)$, which is a $p$-neighbourhood, because reduction modulo $p$ is absolute Frobenius on the model. Let
\begin{equation}\label{eq:psi}
 \psi_{Y_0}:E_p(Y_0,G)\longrightarrow E_p(Y_0,G)
\end{equation}
be the resulting endomorphism. All affine coordinates, including homotopy parameters, are raised to their $p$th powers. The operation is $S$-linear after base change from the model.

Coordinate permutations, coordinate diagonals, open inclusions, projections, and the sections at $0$ and $1$ commute with these power germs. The open-cover and induced-neighbourhood diagrams in Section~\ref{sec:quotients} therefore become diagrams equipped with commuting endomorphisms. This assertion only uses equality of the coordinate formulas on common domains; it does not require a natural Frobenius lift for arbitrary schemes over $S$.

\subsection{Linear comparison of lifts and sequential telescopes}
For an object $E$ with an endomorphism $\psi$, write
\begin{equation}\label{eq:telescope}
 \operatorname{Tel}_{\psi}E
 =\operatorname*{colim}\bigl(E\xrightarrow{\psi}E
 \xrightarrow{\psi}E\longrightarrow\cdots\bigr).
\end{equation}
The induced endomorphism on this telescope is invertible. Colimits of diagrams with commuting endomorphisms commute with \eqref{eq:telescope}.

We need a precise way of passing one noncommuting polynomial homotopy to these telescopes. The following observation supplies it without any global coherence assertion about Frobenius lifts.

\begin{lemma}[Comparison of lifts]\label{lem:lift-homotopy}
Let $f,g:Y_0\dashrightarrow Z_0$ be equivariant germs between permutation models with the same reduction modulo $p$. Linear interpolation of their ambient coordinates induces a homotopy between $E_p(f)$ and $E_p(g)$.

For any equivariant germ $f$, this construction gives a specified homotopy
\begin{equation}\label{eq:commutator-homotopy}
 E_p(P_{Z_0}f)\simeq E_p(fP_{Y_0}).
\end{equation}
It is constant on a parameter face where the two coordinate formulas agree exactly, and it factors through a coordinate fixed locus on a face where both formulas do so.
\end{lemma}
\begin{proof}
Choose a common $p$-neighbourhood on which both germs are defined. For disjoint-union models, first restrict to the open-and-closed $p$-neighbourhood where they land in the same target component. On each component, use the same ambient coordinate chart for their common target component and interpolate by
\[
 (1-s)f+s g
\]
in the ambient torus and affine coordinates. The expressions are equivariant because the action permutes coordinates. On the characteristic-$p$ fibre they agree with $f=g$, independently of $s$. Thus the open where they land in $Z_0$ contains the entire characteristic-$p$ fibre of the source times $\mathbb A^1$. Lemma~\ref{lem:p-neighbourhood}, with its parameter version, and $\mathbb A^1$-invariance give the homotopy. The relative assertions follow from the same coordinate interpolation: it is constant when $f=g$ and preserves the equal-coordinate equations of a fixed locus.

Finally, $P_{Z_0}f$ and $fP_{Y_0}$ have the same reduction, by naturality of absolute Frobenius over $\mathbb F_p$. This proves \eqref{eq:commutator-homotopy}.
\end{proof}

\begin{lemma}[Sequential compatibility]\label{lem:sequential-compatibility}
Let $E$ and $F$ be objects of a stable $\infty$-category admitting sequential colimits, with endomorphisms $\psi_E$ and $\psi_F$. A map $f:E\to F$ together with a specified homotopy $f\psi_E\simeq\psi_F f$ determines a map
\[
 \operatorname{Tel}_{\psi_E}E\longrightarrow\operatorname{Tel}_{\psi_F}F.
\]
A factorization of the map and its specified square through a third such diagram gives the corresponding factorization after taking telescopes.
\end{lemma}
\begin{proof}
The specified square is a morphism from $f$ to itself in the arrow $\infty$-category. Iterating this morphism gives a sequence in that category. A string of composable arrows extends to a functor on the free category $0\to1\to2\to\cdots$, with a contractible space of choices. Its source and target are the desired sequential diagrams. Colimits in the arrow category are computed pointwise, which proves both assertions. The construction can also be made in a functor category, so it respects a finite diagram of maps and their specified squares.
\end{proof}

In the application below, the relative clauses of Lemma~\ref{lem:lift-homotopy} identify the composites with the endpoint sections. We do not require a system of Frobenius-commutation homotopies natural for all maps and all compositions.

\begin{proposition}[Arithmetic permutation lemma]\label{prop:arithmetic}
For every permutation model $Y_0$ with a finite group $G$ acting through coordinate permutations,
\[
 \operatorname{Tel}_{\psi} E_p(Y_0,G)=0.
\]
In particular, for $X=\mathbb G_{m,S}^{r}$, coordinatewise $p$th powers induce
\begin{equation}\label{eq:arithmetic-degree}
 \operatorname{Tel}_{u_p}
 i_p^*\operatorname{cofib}\left(
 M_S(X^n)_{h\Sigma_n}\longrightarrow
 M_S(h_{\operatorname{Sym}_S^nX})\right)=0.
\end{equation}
\end{proposition}
\begin{proof}
A finite constant kernel acting trivially may first be removed: it does not change the geometric quotient, and its rational homotopy orbits act as the identity. We may therefore assume the coordinate action faithful. Induct simultaneously for all models and auxiliary affine factors on the maximum order of a geometric stabilizer of the model. Free actions have zero error by \'etale descent. The empty model presents no problem.

For the induction step, the constructions in \eqref{eq:maximal-fixed}--\eqref{eq:normalizer-neighbourhood} are defined over $\mathbb Z$. They commute with power germs: their maps are open restrictions, coordinate permutations, or induction on finitely many components. Lemma~\ref{lem:error-excision}, Lemma~\ref{lem:p-neighbourhood}, and the induction hypothesis therefore reduce us to
\[
 H\triangleleft N,\qquad Z=V^H,
\]
with $Z$ the maximal-stabilizer locus. Equation~\eqref{eq:fixed-error-zero} gives $E_p(Z,N)=0$ before telescoping.

For an $H$-orbit $I\subseteq\{1,\ldots,n\}$, put
\[
 d_I=|I|,\qquad a_I=p^{v_p(d_I)},\qquad
 b_I=\binom{d_I}{a_I}.
\]
The integer $b_I$ is prime to $p$. Indeed, writing $d_I=a_I d'_I$ with $p\nmid d'_I$, the identity
\[
 (1+T)^{d_I}=(1+T^{a_I})^{d'_I}\quad\text{over }\mathbb F_p
\]
shows that $b_I\equiv d'_I\pmod p$. Inverting the finitely many $b_I$ is restriction to a $p$-neighbourhood and hence harmless by Lemma~\ref{lem:p-neighbourhood}.

Choose $q=p^e$ divisible by all $a_I$. For each torus coordinate $j$, set
\begin{equation}\label{eq:arithmetic-retraction}
 R_{I,j}(x)=
 \left(b_I^{-1}e_{a_I}(x_{a,j}:a\in I)\right)^{q/a_I}.
\end{equation}
On the fixed locus, where $x_{a,j}=t_{I,j}$ for $a\in I$, this is the exact identity
\begin{equation}\label{eq:diagonal-identity}
 R_{I,j}=t_{I,j}^{q}
\end{equation}
over the localized integer ring. Since $N$ permutes $H$-orbits of the same size, these formulas are $N$-equivariant.

Replace all coordinates in an orbit $I$ by $R_{I,j}$, and each auxiliary affine coordinate $z_l$ by $z_l^q$. This defines a polynomial map to the ambient fixed locus. Restrict to an invariant open $V'\subset V$ on which its torus coordinates are units, its image lies in $Z$, and coordinatewise $q$th powers land in $V$. By \eqref{eq:diagonal-identity}, $V'$ contains $Z_{\mathbb F_p}$. Denote the resulting map by $\rho:V'\to Z$.

The inclusion $V'\hookrightarrow V$ induces an equivalence on telescoped errors. To see this, the open
\[
 V'\cup(V\setminus Z)
\]
is a $p$-neighbourhood of $V$, while $V\setminus Z$ and its intersection with $V'$ have smaller stabilizers. Apply Lemma~\ref{lem:p-neighbourhood}, Mayer--Vietoris, and induction. This argument does not require $V'$ to contain every characteristic-zero point of $Z$.

Consider the $N$-equivariant interpolation
\begin{equation}\label{eq:arithmetic-homotopy}
 \mathcal H_\lambda(x_{a,j})
 =(1-\lambda)x_{a,j}^{q}+\lambda R_{I,j}(x),
 \qquad a\in I,
 \qquad
 \mathcal H_\lambda(z_l)=z_l^q.
\end{equation}
Let $W\subset V'\times\mathbb A^1$ be the open on which this is a map to $V$. It contains both complete endpoint fibres and $(V'\cap Z)\times\mathbb A^1$, since the homotopy is constant on the fixed locus. Therefore
\[
 V'\times\mathbb A^1
 =W\cup\bigl((V'\setminus Z)\times\mathbb A^1\bigr).
\]
The second member and the intersection have smaller stabilizers. All structural arrows commute with power germs. Taking telescopes in this Mayer--Vietoris square gives an equivalence
\begin{equation}\label{eq:W-projection}
 \operatorname{Tel}_{\psi}E_p(W,N)
 \xrightarrow{\simeq}
 \operatorname{Tel}_{\psi}E_p(V',N)
\end{equation}
induced by projection. The endpoint sections are both inverses to it and hence agree.

It remains to justify composing these sections with \eqref{eq:arithmetic-homotopy} on telescopes. Apply Lemma~\ref{lem:lift-homotopy} to the two germs
\[
 P_V\mathcal H,\qquad \mathcal H P_W.
\]
They have the same reduction modulo $p$; on the source, $P_W$ raises $\lambda$ as well as the other coordinates to their $p$th powers. Their coordinate interpolation provides the adjacent-square homotopy needed in Lemma~\ref{lem:sequential-compatibility}. On $\lambda=0$ the formulas agree exactly, so this homotopy is constant. On $\lambda=1$ both maps and their interpolation land in $Z$. Apply Lemma~\ref{lem:sequential-compatibility} to the finite diagram consisting of $\mathcal H$ and its two endpoint restrictions; the relative properties just checked supply its specified squares. The induced map on telescopes therefore has the following endpoint composites: at $0$ it is the canonical $q$th-power map $V'\dashrightarrow V$, and at $1$ it factors through the zero object $E_p(Z,N)$.

Since the sections in \eqref{eq:W-projection} agree, the $q$th-power map on these telescopes is zero. It is also the composite of the equivalence induced by $V'\hookrightarrow V$ with $\psi^e$, which is invertible on a telescope. The telescope of $E_p(V,N)$ is therefore zero. This completes the induction.
\end{proof}

\begin{remark}\label{rem:no-wild-purity}
Proposition~\ref{prop:arithmetic} makes no assertion that an unlocalized wild quotient comparison is an equivalence. It proves vanishing only after a power-map telescope. The proof uses the smooth coordinate fixed loci of permutation actions and a homotopy of the full comparison error. It does not use wild equivariant purity or an identification of a quotient normal bundle. A map in \eqref{eq:arithmetic-homotopy} may increase stabilizers, which is permissible for the functor \eqref{eq:error}.
\end{remark}

\section{The graph comparison and seminormal descent}\label{sec:graph}
\subsection{The divided monoid comparison}
\begin{theorem}\label{thm:divided-comparison}
Let $S$ be a seminormal finite-type $\mathbb Z$-scheme and $X=\mathbb G_{m,S}^{r}$. The canonical morphism
\[
 M_S(P_X^{\mathrm{div}})\longrightarrow M_S(D_X^{\mathrm{div}})
\]
is an equivalence.
\end{theorem}
\begin{proof}
Let $K_S$ be its cofiber. Coordinate power maps act on its defining arrow. For every prime $p$, Proposition~\ref{prop:arithmetic}, followed by the rational homotopy-orbit-to-orbit comparison, gives
\[
 \operatorname{Tel}_{u_p}\, i_p^*K_S=0.
\]
Indeed, take coproducts over the symmetric degrees and then the filtered colimit dividing multiplicities. Power maps commute with addition and division of multiplicities, and all these colimits commute with pullback and cofibers.

By Lemma~\ref{lem:power-free}, $u_p$ is already an equivalence on both terms defining $K_S$, and hence on $K_S$ and $i_p^*K_S$. Their power-map telescopes are thus redundant. We obtain
\begin{equation}\label{eq:K-special-zero}
 i_p^*K_S=0\qquad\text{for every prime }p.
\end{equation}

Let $S_{\mathbb Q}=S\times_{\mathbb Z}\mathbb Q$. In symmetric degree $n$, Proposition~\ref{prop:tame-permutation} proves the quotient comparison over the open $S[1/n!]$. Open restriction of a restricted representable is the restricted representable of the open base change. It follows that the pullback of each fixed-degree error to $S_{\mathbb Q}$ is zero. Taking the same coproducts and multiplicity colimits gives
\begin{equation}\label{eq:K-generic-zero}
 K_S|_{S_{\mathbb Q}}=0.
\end{equation}
This reasoning only uses open base change after the error has vanished; it does not commute a singular restricted representable with an arbitrary closed base change.

Every point of $S$ factors through $S_{\mathbb Q}$ or some $S_p$. Point pullbacks are jointly conservative for $\mathrm{DA}_{\mathrm{\acute{e}t}}$, including on noncompact objects: combine its identification with Beilinson motives with \cite[Proposition~4.3.17 and Proposition~15.1.6]{CD}. Equations~\eqref{eq:K-special-zero} and~\eqref{eq:K-generic-zero} imply $K_S=0$.
\end{proof}

\begin{corollary}\label{cor:seminormal-graph}
For $S$ as in Theorem~\ref{thm:divided-comparison} and every $r\geq0$, the canonical graph map
\[
 M_S(\mathbb Q[\mathbb G_{m,S}^{r}])
 \longrightarrow M_S(L_S(\mathbb G_{m,S}^{r})_{\mathbb Q})
\]
is an equivalence.
\end{corollary}
\begin{proof}
The monoid $P_X^{\mathrm{div}}$ is cancellative, and its group completion is the free rational sheaf $\mathbb Q[X]$. By Proposition~\ref{prop:cycle-presentation}, $D_X^{\mathrm{div}}\simeq\mathcal C_S(X)$, which is also cancellative. Lemma~\ref{lem:cycle-completion} identifies its sheafified group completion with $L_S(X)_{\mathbb Q}$. Apply Lemma~\ref{lem:group-completion} to Theorem~\ref{thm:divided-comparison}. All maps come from the passage from points to their graph cycles, so the resulting equivalence is the stated canonical comparison.
\end{proof}

\subsection{Finite subintegral descent on underlying spaces}
A morphism is \emph{subintegral} if it is integral, bijective on points, and induces isomorphisms on residue fields. We need a direct-image statement on the smooth sites, not just an abstract invariance theorem for stable categories.

\begin{lemma}\label{lem:subintegral}
Let $q:T\to S$ be a finite subintegral morphism. Presheaf direct image is
\[
 q_*F(U)=F(U\times_S T),\qquad U\in\mathrm{Sm}_S.
\]
Pullback is an equivalence on $S^1$-stable Nisnevich motivic categories, with inverse induced by this direct image. It gives an equivalence on rational \'etale stable motives, and for every sheaf of spaces $F$ on $\mathrm{Sm}_T$ there is a natural equivalence
\begin{equation}\label{eq:subintegral-underlying}
 M_S(q_*F)\simeq q_*M_T(F).
\end{equation}
The direct image on the right is the inverse of motivic pullback.
\end{lemma}
\begin{proof}
Begin with pointed presheaves that send finite disjoint unions to products. Direct image preserves sifted colimits objectwise. A universal homeomorphism identifies the open-and-closed decompositions of $U$ and $U_T$, so this direct image also preserves finite coproducts, and hence all colimits in these presheaf categories. It preserves finite products and smash products. These properties also follow from the universal-homeomorphism case of the norm construction \cite[Theorem~3.3(5)]{BH}.

Because $q$ is integral, direct image preserves Nisnevich and motivic equivalences \cite[Proposition~2.11]{BH}. It therefore induces a colimit-preserving exact right adjoint to pullback after $S^1$-stabilization.

For a closed immersion $i:Z\hookrightarrow S$ with complementary open $j:U\hookrightarrow S$, ordinary motivic localization gives the triangle
\[
 j_\sharp j^*E\longrightarrow E\longrightarrow i_*i^*E
\]
after pointing and $S^1$-stabilization \cite[Corollary~5]{HoyoisLocalization}. In particular, $i_*$ is fully faithful and $(i^*,j^*)$ is conservative. Stabilization preserves this localization, since integral direct image commutes with the colimits involved.

Decorate the corresponding immersions after base change to $T$ by primes. In pointed presheaves preserving finite disjoint unions there are identities
\[
 q_*j'_\sharp\simeq j_\sharp q_{U*},\qquad
 q_*i'_*\simeq i_*q_{Z*}.
\]
For the first identity, the objectwise extension-by-zero description on a connected smooth test is nontrivial exactly when that test factors through $U$; this condition is equivalent after a universal homeomorphism. General tests follow by finite disjoint unions. The second identity is composition of direct images. Since the functors concerned preserve motivic equivalences, the identities descend to the localized categories. Pushing forward the localization triangle on $T$ and comparing it with the one on $S$ gives
\[
 i^*q_*\simeq q_{Z*}i'^*.
\]
Open base change is immediate as well.

Localization first gives invariance under nilpotent closed immersions, so we may reduce $S$ and $T$. A finite subintegral morphism between reduced noetherian schemes is an isomorphism on a dense open meeting every irreducible component: it is an isomorphism at all generic points, and the finite algebra cokernel is supported on a proper closed subset. Noetherian induction on this closed complement, using the preceding base-change identities and conservativity, proves that the unit and counit of $q^*\dashv q_*$ are equivalences.

The same direct image preserves \'etale local equivalences. The stalk argument in \cite[Proposition~2.11]{BH} applies to strict henselian local tests, since a finite subintegral extension of a strictly henselian local ring is again strictly henselian local. Hence the colimit comparison survives \'etale localization and rational linearization. The inverse equivalence is symmetric monoidal, and pullback identifies the Tate objects; it consequently extends to Tate stabilization. Direct image has commuted with the underlying-space, rational-linearization, and stabilization constructions at each stage, proving \eqref{eq:subintegral-underlying}.
\end{proof}

\begin{theorem}\label{thm:finite-type-graph}
Let $S$ be a finite-type $\mathbb Z$-scheme. For every $r\geq0$, the canonical graph map for $X=\mathbb G_{m,S}^{r}$ induces an equivalence
\[
 M_S(\mathbb Q[X])\xrightarrow{\simeq}M_S(L_S(X)_{\mathbb Q}).
\]
Consequently, $\Phi_S$ is an equivalence.
\end{theorem}
\begin{proof}
The seminormalization $q:T\to S$ is finite and subintegral, since $S$ is excellent. The scheme $T$ is seminormal and finite type over $\mathbb Z$, so Corollary~\ref{cor:seminormal-graph} applies there.

Rational relative-cycle sheaves are qfh sheaves \cite[Theorem~10.5.6]{CD}. Such a sheaf is invariant under a universal homeomorphism: the morphism and its diagonal are qfh covers, and the sheaf equalizer identifies its values. Applying this to $U\times_S T\to U$ gives
\[
 L_S(X)_{\mathbb Q}\simeq q_*L_T(X_T)_{\mathbb Q}.
\]
Here the relative-cycle definition is compatible with the displayed fibre products. Lemma~\ref{lem:subintegral} now gives
\begin{align*}
 M_S(\mathbb Q[X])
 &\simeq q_*M_T(\mathbb Q[X_T])\\
 &\simeq q_*M_T(L_T(X_T)_{\mathbb Q})\\
 &\simeq M_S(q_*L_T(X_T)_{\mathbb Q})
 \simeq M_S(L_S(X)_{\mathbb Q}).
\end{align*}
The first equivalence in this chain is the unit of the motivic equivalence $q^*\dashv q_*$, together with compatibility of pullback with the free rational motive of the smooth scheme $X$. It does not assert equality of the underlying free-point sheaves before motivic localization. Every identification is compatible with graph maps. This proves the graph comparison, and Proposition~\ref{prop:torus-criterion} gives $\Phi_S$.
\end{proof}

\section{The rational comparison and the main theorem}\label{sec:conclusion}
\subsection{Finite-type arithmetic bases}
\begin{proposition}\label{prop:base-change-formal}
Suppose $\Phi_S$ and $\Phi_T$ are equivalences. Then, for any morphism $f:T\to S$, the canonical exchange map $f^*H\mathbb Q_S\to H\mathbb Q_T$ is an equivalence.
\end{proposition}
\begin{proof}
Naturality gives the commutative square of ring spectra
\[
\begin{tikzcd}
 f^*\mathbb S_{S,\mathbb Q,+}\arrow[r,"f^*\Phi_S"]\arrow[d,"\simeq"']
 &f^*H\mathbb Q_S\arrow[d]\\
 \mathbb S_{T,\mathbb Q,+}\arrow[r,"\Phi_T"]&H\mathbb Q_T.
\end{tikzcd}
\]
The left vertical map is an equivalence because the rational plus sphere commutes with pullback. The two horizontal maps are equivalences by assumption, so the right vertical map is one as well.
\end{proof}

\begin{proposition}\label{prop:finite-type-category}
The canonical comparison
\[
 \mathrm{DA}_{\mathrm{\acute{e}t}}(S,\mathbb Q)
 \longrightarrow\mathrm{DM}(S,\mathbb Q)
\]
is a symmetric monoidal equivalence when $S$ is finite type over $\mathbb Z$. It is compatible with pullback between such schemes.
\end{proposition}
\begin{proof}
Theorem~\ref{thm:finite-type-graph} applies to every finite-type $\mathbb Z$-scheme, and Proposition~\ref{prop:base-change-formal} gives base-change compatibility of $H\mathbb Q$ on this entire family. Apply the module-comparison criterion \cite[Proposition~11.4.7]{CD} on the full category of finite-type $\mathbb Z$-schemes. This category is stable under the finite-type constructions used in that criterion, and the base-change compatibility required there has just been established on this entire category. It identifies $\mathrm{DM}(S,\mathbb Q)$ with $H\mathbb Q_S$-modules. Since $H\mathbb Q_S\simeq\mathbb S_{S,\mathbb Q,+}$, Proposition~\ref{prop:rational-sphere} identifies this module category with $\mathrm{DA}_{\mathrm{\acute{e}t}}(S,\mathbb Q)$. All functors are the canonical ones induced by adding transfers and are compatible with pullback.
\end{proof}

\subsection{Weak continuity}
\begin{proof}[Proof of Theorem~\ref{thm:rational-main}]
First let $S=\operatorname{Spec}A$ be affine. Express $A$ as the filtered union of its finitely generated $\mathbb Z$-subalgebras $A_i$, and put $S_i=\operatorname{Spec}A_i$. Every $S_i$ is excellent and finite-dimensional. The transitions are affine and dominant, since the ring inclusions are injective; the limit $S$ is noetherian and finite-dimensional by hypothesis.

Use a subscript $c$ for compact objects. Weak continuity of $\mathrm{DM}$ gives
\[
 \operatorname*{colim}_i\mathrm{DM}_c(S_i,\mathbb Q)
 \simeq\mathrm{DM}_c(S,\mathbb Q)
\]
by \cite[Theorem~11.1.24, Proposition~4.3.4 and Example~11.1.25]{CD}; compact and constructible motives agree by \cite[Theorem~11.1.13]{CD}. Similarly, the known comparison with Beilinson motives and continuity give
\[
 \operatorname*{colim}_i
 \mathrm{DA}_{\mathrm{\acute{e}t},c}(S_i,\mathbb Q)
 \simeq\mathrm{DA}_{\mathrm{\acute{e}t},c}(S,\mathbb Q)
\]
by \cite[Propositions~15.1.4 and~15.1.6, Theorem~16.2.18]{CD}. The colimits are understood in the idempotent-complete stable sense, or as the corresponding $2$-colimits of compact triangulated categories.

The equivalences of Proposition~\ref{prop:finite-type-category} commute with the transition functors. They therefore give an equivalence on compact objects over $S$. The canonical comparison preserves coproducts and sends the compact smooth Tate generators to their transferred counterparts. Compact generation extends the equivalence to the full categories \cite[Corollary~1.3.20]{CD}.

For general $S$, choose a finite affine open cover. Both constructions satisfy Zariski descent. If $F_S$ is the comparison and $R_S$ its right adjoint, compatibility of $F$ with smooth pushforward gives
\[
 j^*R_S\simeq R_Uj^*
\]
for an open immersion $j:U\hookrightarrow S$, by adjunction. The unit and counit consequently restrict to those on the affine opens, where they are equivalences. Joint conservativity of these restrictions proves the result. The remaining equivalences in the theorem are Proposition~\ref{prop:rational-sphere} and the established Beilinson comparison.
\end{proof}

\begin{proof}[Proof of Corollary~\ref{thm:ring-main}]
By Theorem~\ref{thm:rational-main}, the right adjoint of
\[
 \varphi_S^*: \mathrm{SH}(S)\longrightarrow\mathrm{DM}(S,\mathbb Q)
\]
is identified with inclusion of the rational plus part. Its value on the unit is therefore $\mathbb S_{S,\mathbb Q,+}$. By the definition of $H\mathbb Q_S$, this identifies the canonical ring map $\Phi_S$ with an equivalence, not just its underlying objects. Module categories then agree by Proposition~\ref{prop:rational-sphere}, and arbitrary base-change compatibility follows from Proposition~\ref{prop:base-change-formal}.
\end{proof}

\subsection{The comparison with arbitrary coefficients}
\begin{proof}[Proof of Theorem~\ref{thm:main}]
Theorem~\ref{thm:rational-main} gives the rational comparison with Nisnevich motives with transfers. By \cite[Proposition~2.2.10]{CDEt}, rational Nisnevich and \'etale motives with transfers agree. Hence the canonical functor $F_{S,\mathbb Q}$ is an equivalence for every noetherian finite-dimensional $S$. Proposition~\ref{prop:integral-reduction} applies and proves that $F_{S,R}$ is an equivalence for every commutative ring $R$.

The graph functors form a symmetric monoidal family compatible with pullback in the base and with extension of coefficients. The equivalences just proved are those functors themselves, so their stated compatibilities follow without choosing new comparison maps or inverse equivalences. This completes the proof of the main theorem.
\end{proof}

\begin{remark}[Scope of the geometric statement]\label{rem:scope}
The theorem includes equal and mixed characteristics. The geometric comparison proved along the way is
\[
 M_S(P_X^{\mathrm{div}})\simeq M_S(D_X^{\mathrm{div}})
\]
for split tori over seminormal finite-type arithmetic bases. Before division of multiplicities, only the power-telescoped assertion of Proposition~\ref{prop:arithmetic} is used in wild characteristic. This geometric argument makes no assertion about arbitrary finite wild quotients. The extension to integral and general coefficients is the separate \'etale argument of Section~\ref{sec:coefficients}. With our notation, the rational $\mathbb A^1$-derived category satisfies
$\mathrm{DA}(S,\mathbb Q)\simeq\mathrm{SH}(S)_{\mathbb Q}$; it is its plus part that agrees with $\mathrm{DM}(S,\mathbb Q)$.
\end{remark}

\end{document}